\documentclass[11pt]{amsart}
\usepackage[margin=1.15in]{geometry}
\usepackage{amssymb,amsmath,mathtools}
\usepackage{booktabs}
\usepackage{enumitem}
\usepackage[expansion=false]{microtype}
\usepackage{tikz}
\usetikzlibrary{arrows.meta}
\usepackage{pgfplots}
\pgfplotsset{compat=1.18}
\usepackage[hidelinks]{hyperref}
\hypersetup{pdftitle={Retriangular numbers and affine maps between polygonal numbers},pdfauthor={Mateo Matijasevick}}

\newtheorem{theorem}{Theorem}[section]
\newtheorem{proposition}[theorem]{Proposition}
\newtheorem{lemma}[theorem]{Lemma}
\newtheorem{corollary}[theorem]{Corollary}
\newtheorem{introthm}{Theorem}

\theoremstyle{definition}
\newtheorem{example}[theorem]{Example}
\theoremstyle{remark}
\newtheorem{remark}[theorem]{Remark}

\numberwithin{equation}{section}

\newcommand{\N}{\mathbb{N}}
\newcommand{\Z}{\mathbb{Z}}
\newcommand{\Q}{\mathbb{Q}}
\newcommand{\R}{\mathbb{R}}
\newcommand{\cT}{\mathcal{T}}
\newcommand{\cP}{\mathcal{P}}
\newcommand{\cS}{\mathcal{S}}
\newcommand{\cN}{\mathcal{N}}
\DeclareMathOperator{\End}{End}
\DeclareMathOperator{\Aut}{Aut}
\DeclareMathOperator{\id}{id}

\definecolor{vizblue}{HTML}{2A78D6}
\definecolor{vizorange}{HTML}{EB6834}
\definecolor{vizgray}{HTML}{52514E}

\pgfplotsset{
  parabolaaxis/.style={
    axis lines=left,
    axis line style={-{Stealth[length=3pt]}, black!55},
    width=0.5\textwidth, height=0.52\textwidth,
    xmin=-6.9, xmax=6.0, ymin=-1.6, ymax=17.2,
    xtick={-6,-4,-2,0,2,4}, ytick={0,5,10,15},
    tick style={black!55},
    tick label style={font=\scriptsize, text=black!75},
    xlabel={$x$}, ylabel={$y$},
    every axis x label/.style={at={(ticklabel* cs:1)}, anchor=west, font=\scriptsize},
    every axis y label/.style={at={(ticklabel* cs:1)}, anchor=south, font=\scriptsize},
    clip=false,
  },
  latticepts/.style={only marks, mark=*, mark size=1.3pt, black},
}

\title[Retriangular numbers]{Retriangular numbers and affine maps\\ between polygonal numbers}
\author{Mateo Matijasevick}
\address{Purdue University}
\email{mmatijas@purdue.edu}

\begin{document}

\begin{abstract}
A positive integer $a$ is called \emph{retriangular} if there is a positive integer $b$ such that $aT+b$ is a triangular number whenever $T$ is. We show that the retriangular numbers are exactly the odd squares greater than $1$, and that for each of them $b$ is unique. The proof rests on a rigidity principle: a non-constant affine map carrying the values of one quadratic sequence into the values of another is induced by an affine map of the plane sending the vertex of one parabola to the vertex of the other. We use it to classify all affine maps between sets of polygonal numbers; the $s$-gonal numbers map affinely into the $t$-gonal numbers if and only if $M_t\mid M_s$, where $M_s=2(s-2)/\gcd(s,4)$ defines a bijection from $\{3,4,5,\dots\}$ onto the positive integers. We also treat sparse data, polynomial maps, figurate numbers of higher degree and sums of two triangular numbers, and we study the monoid of affine self-maps of an arbitrary set of integers. For the $s$-gonal numbers this monoid is an arithmetical congruence monoid, and the factorization of self-maps into irreducible ones is unique only for triangular and square numbers.
\end{abstract}

\maketitle

\section{Introduction}\label{sec:intro}

Let $T_n=\frac{n(n+1)}{2}$, and let $\cT=\{T_n:\ n\in\N\}$ be the set of triangular numbers. Since $8T_n+1=(2n+1)^2$, a fact quoted by Plutarch and used by Diophantus \cite[p.~84]{HeathHGM}, a positive integer $N$ is triangular if and only if $8N+1$ is a perfect square. Euler observed that $9T+1$, $25T+3$, $49T+6$ and $81T+10$ are triangular whenever $T$ is \cite[\S87, pp.~264--265]{Euler}; more generally (see \cite[(4.5)]{HoggattBicknell} and \cite{OEIS}),
\begin{equation}\label{eq:euler}
(2c+1)^2\,T_n+T_c=T_{(2c+1)n+c}\qquad(c\ge0,\ n\in\Z).
\end{equation}
We call a positive integer $a$ \emph{retriangular} if there is a positive integer $b$ such that $aT+b$ is triangular for every triangular number $T$. Our first result (Theorem~\ref{thm:main}) shows that \eqref{eq:euler} accounts for all of them: the retriangular numbers are exactly the odd squares greater than $1$, and for $a=(2c+1)^2$ the only admissible value of $b$ is $T_c$. This question was proposed by the author as a problem for the 2020 Colombian Mathematical Olympiad. A related question is Problem~B6 of the 1988 Putnam Competition, which asks for infinitely many pairs of integers $(a,b)$ such that, for every positive integer $t$, the number $at+b$ is triangular if and only if $t$ is \cite{KPV}; there, $0$ also counts as a triangular number. A pair with this property has $a>0$, because $at+b$ is triangular for $t=1$ and for all large triangular $t$, but not for $t=2$. Hence, by Corollary~\ref{cor:strong} and Proposition~\ref{prop:reflection}, the pairs with this property are exactly the pairs $\bigl((2c+1)^2,T_c\bigr)$ with $c\ge0$.

The mechanism behind this is geometric. If a non-constant affine map $N\mapsto aN+b$ sends the values of a quadratic polynomial $f$ at the positive integers into the values of another quadratic polynomial $g$, then it is induced by an affine map of the plane that carries the graph of $f$ onto the graph of $g$ and the vertex of one parabola to the vertex of the other (Proposition~\ref{prop:rigidity} and Section~\ref{sec:geometry}). For the $s$-gonal numbers
\[
P_s(n)=\frac{(s-2)n^2-(s-4)n}{2},\qquad \cP_s=\{P_s(n):\ n\in\N\}\qquad(s\ge3),
\]
this leads to a complete classification. Let
\[
M_s=\frac{2(s-2)}{\gcd(s,4)},
\]
which is the denominator of the abscissa $\frac{s-4}{2(s-2)}$ of the vertex of $P_s$.

\begin{introthm}\label{thmA}
Let $s\ge3$, $a>0$ and $b\in\R$. Then $aP+b\in\cP_s$ for every $P\in\cP_s$ if and only if $a=x^2$ for some $x\in\N$ with $x\equiv1\pmod{M_s}$ and
\[
b=\frac{(s-4)^2(x^2-1)}{8(s-2)}.
\]
\end{introthm}

For instance, $49P+2$ is pentagonal whenever $P$ is, and $16P+5$ is octagonal whenever $P$ is. For $s\ne4$ the analogues of the retriangular numbers are thus the squares $x^2$ with $x>1$ and $x\equiv1\pmod{M_s}$, while for the squares ($s=4$) the intercept is forced to be $0$. For odd $s$, the pairs of positive integers $(a,b)$ with $a\cP_s+b\subseteq\cP_s$ are also the subject of a computational problem of Project Euler \cite{PE647}. Maps between different families are governed by divisibility.

\begin{introthm}\label{thmB}
Let $s,t\ge3$. There exist $a>0$ and $b\in\R$ with $a\cP_s+b\subseteq\cP_t$ if and only if $M_t\mid M_s$. Moreover, $s\mapsto M_s$ is a bijection from $\{3,4,5,\dots\}$ onto $\N$.
\end{introthm}

Thus the polygonal families, preordered by affine embeddability, form a copy of the positive integers ordered by divisibility, with the order reversed; all the maps are described explicitly in Theorem~\ref{thm:between}. Every family maps into the squares (by Diophantus' identity), a family maps into the triangular numbers exactly when $s\not\equiv0\pmod4$, and the only inclusion between two families is the classical one, $\cP_6\subset\cP_3$.

We then consider arbitrary sets of integers. For $S\subseteq\Z$, the affine maps $\phi$ with $\phi(S)\subseteq S$ form a monoid under composition, whose non-constant elements we call the \emph{self-maps} of $S$. When $\lvert S\rvert\ge2$, self-maps have integral slopes and are injective, and the surjective ones form a subgroup of the infinite dihedral group (Proposition~\ref{prop:basic}). For sparse sets they are rigid.

\begin{introthm}\label{thmC}
Let $S\subseteq\Z$ be infinite and bounded below, and assume that the differences between consecutive elements of $S$ tend to infinity. Then any two self-maps of $S$ commute, and each self-map is determined by its slope. If $S$ has a self-map other than the identity, there is a rational number $\omega\le\min S$ such that every self-map of $S$ is a homothety $x\mapsto\omega+a(x-\omega)$ with $a\in\N$.
\end{introthm}

For the polygonal numbers the centre $\omega$ is the minimum of $P_s$ on $\R$; this is why the intercept in Theorem~\ref{thmA} is forced by the slope. Composition of self-maps of $\cP_s$ corresponds to multiplication of the ratios $x$ in Theorem~\ref{thmA}, which raises a question of factorization.

\begin{introthm}\label{thmD}
Let $s\ge3$. Under composition, the self-maps of $\cP_s$ form a monoid isomorphic to the multiplicative monoid $H_{M_s}=\{x\in\N:\ x\equiv1\pmod{M_s}\}$. Every self-map is a composition of irreducible self-maps. This factorization is unique for every self-map if and only if $s\in\{3,4\}$, and all factorizations of each self-map have the same number of factors if and only if $s\in\{3,4,5,6,8\}$.
\end{introthm}

For hexagonal numbers, for example, $E_{6,9}\circ E_{6,49}=E_{6,21}\circ E_{6,21}$ with all four factors irreducible (Example~\ref{ex:hilbert}), where $E_{s,x}$ denotes the self-map of $\cP_s$ with index ratio $x$; this is Hilbert's classical example $441=9\cdot49=21^2$.

We also prove several complements. Very sparse data already force the conclusion of Theorem~\ref{thm:main}: the set of $n$ with $aT_n+b\in\cT$ is all of $\N$, or finite, or has only $O(\log N)$ elements up to $N$ (Theorem~\ref{thm:trichotomy}). We describe all polynomials $P$ with $P(\cT)\subseteq\cT$ (Theorem~\ref{thm:polymaps}); we show that in degree $d\ge3$ only translated monomials $\lambda(n-\nu)^d+\mu$ admit nontrivial affine self-maps (Theorem~\ref{thm:higher}); and we show that the affine self-maps of the set of sums of two triangular numbers are the maps $N\mapsto aN+\frac{a-1}{4}$ with $a\equiv1\pmod4$ a sum of two squares (Theorem~\ref{thm:two}), so that for two triangular numbers the multipliers are norms from $\Z[i]$ rather than squares.

\subsection*{Related work}
Identities between figurate numbers such as \eqref{eq:euler} are classical; see \cite[Ch.~I]{Dickson} and \cite{Deza}, and \cite{Heath} for Diophantus' treatise on polygonal numbers. We have not found a proof of the necessity part of Theorem~\ref{thm:main}, or of Theorems~\ref{thmA} and~\ref{thmB}, in the literature. Proposition~\ref{prop:rigidity} is an explicit quadratic instance of a general principle: if $f(n)$ is a value of the polynomial $g$ for all $n$ in a set that is not too sparse, then $f=g\circ h$ for some polynomial $h$. This is a consequence of Hilbert's irreducibility theorem, in the form that a thin set of integers has only $O(N^{1/2})$ elements in $[-N,N]$ \cite{SerreTopics,Schinzel}; see also \cite{DLS} for value-set problems of this kind. When the equation $f(x)=g(y)$ has infinitely many integer solutions is decided by the Bilu--Tichy theorem \cite{BiluTichy}; Theorem~\ref{thm:trichotomy} makes explicit, in the simplest case, the difference between ``infinitely many $x$'' and ``all $x$''. The equation $aP_s(n)=P_s(m)$, required for some rather than for all $n$, was studied by Chahal, Griffin and Priddis \cite{CGP} by means of Pell equations. The functional equations $P\circ f=g\circ h$ of Section~\ref{sec:polymaps} belong to Ritt's theory of polynomial decomposition \cite{Ritt}. For Krull monoids, arithmetical congruence monoids and Carlitz's theorem we refer to \cite{BaginskiChapman,BCCM,Carlitz,CLMZ,GHK}. The factorization properties of the monoids $H_M$ that enter Theorem~\ref{thmD} are known \cite{BaginskiChapman,HalterKoch,JamesNiven}; what is new there is the isomorphism between $H_{M_s}$ and the monoid of self-maps of $\cP_s$. Sets of integers generated by affine maps were studied by Klarner and Rado \cite{KlarnerRado}.

\subsection*{Organization and notation}
Section~\ref{sec:triangular} proves Theorem~\ref{thm:main}. Section~\ref{sec:rigidity} establishes the rigidity principle and its geometric interpretation. Sections~\ref{sec:polygonal} and~\ref{sec:between} prove Theorems~\ref{thmA} and~\ref{thmB}, Section~\ref{sec:further} contains the complements mentioned above, and Section~\ref{sec:general} studies the self-maps of arbitrary sets and proves Theorems~\ref{thmC} and~\ref{thmD}.

Throughout, $\N=\{1,2,3,\dots\}$, and $T_n=\frac{n(n+1)}{2}$ for all $n\in\Z$, so that $T_{-1-n}=T_n$. An affine map is a map $x\mapsto ax+b$ with $a,b\in\R$, and for $S\subseteq\R$ we write $aS+b=\{as+b:\ s\in S\}$.

\subsection*{Formal verification}
All results of this paper, together with the precise mathematical claims made in the remarks, examples, tables and figure captions, have been formalized and checked in the Lean~4 proof assistant with the Mathlib library, in 90 formal statements; the formal proofs depend only on the standard axioms of Lean. The claims that are not formalized are the facts quoted from the literature, such as Gauss's theorem on sums of three triangular numbers (Remark~\ref{rem:two}) and the theory of Krull monoids used in the first proof of Theorem~\ref{thm:factorization}. For Lemma~\ref{lem:squares} and Theorems~\ref{thm:higher} and~\ref{thm:factorization} the formalization follows more elementary routes than the proofs given here; for Theorem~\ref{thm:factorization} it is the argument given after its proof, which avoids Krull monoids. The source files are available as ancillary files with the arXiv version of this paper.

\section{Triangular numbers}\label{sec:triangular}

We use the criterion
\begin{equation}\label{eq:plutarch}
N\in\cT\iff N\ge1\ \text{ and }\ 8N+1\ \text{ is a perfect square}\qquad(N\in\Z).
\end{equation}
Indeed $8T_m+1=(2m+1)^2$; conversely, if $8N+1=k^2$ with $k\ge0$, then $k$ is odd, $N=T_{(k-1)/2}$, and $N\ge1$ forces $k\ge3$.

\begin{theorem}\label{thm:main}
Let $a,b\in\N$. Then $aT+b\in\cT$ for every $T\in\cT$ if and only if $a=(2c+1)^2$ and $b=T_c$ for some integer $c\ge1$. In particular, the retriangular numbers are the odd squares greater than $1$, and $b$ is determined by $a$.
\end{theorem}

\begin{proof}
\emph{Sufficiency.} Let $a=(2c+1)^2$ and $b=T_c$. Since $8T_c+1=a$, for every $n\in\Z$
\[
8\,(aT_n+T_c)+1=a\,(8T_n+1)=\bigl((2c+1)(2n+1)\bigr)^2=\bigl(2[(2c+1)n+c]+1\bigr)^2,
\]
which is \eqref{eq:euler}; in particular $a\cT+b\subseteq\cT$.

\emph{Necessity.} Suppose that $a\cT+b\subseteq\cT$. For each $n\in\N$ choose $m_n\in\N$ with $aT_n+b=T_{m_n}$, and put
\[
u_n=2n+1,\qquad k_n=2m_n+1,\qquad h=8b+1-a.
\]
Multiplying by $8$ and adding $1$ gives
\begin{equation}\label{eq:pell}
k_n^2=a\,u_n^2+h\qquad(n\in\N).
\end{equation}
Since $k_n+\sqrt a\,u_n>u_n>2n$, \eqref{eq:pell} yields
\begin{equation}\label{eq:asymp}
k_n-\sqrt a\,u_n=\frac{h}{k_n+\sqrt a\,u_n}=O\!\left(\frac1n\right)\qquad(n\to\infty);
\end{equation}
indeed, the absolute value of the left-hand side is at most $\lvert h\rvert/(2n)$. Subtracting \eqref{eq:asymp} for $n$ from \eqref{eq:asymp} for $n+1$, and using $u_{n+1}-u_n=2$, we get
\[
k_{n+1}-k_n=2\sqrt a+O\!\left(\frac1n\right).
\]
The left-hand side is an integer, and a convergent sequence of integers is eventually constant; hence $2\sqrt a\in\Z$. Thus $4a$ is a perfect square, and therefore so is $a$: $a=x^2$ with $x\in\N$. Now $k_n-xu_n$ is an integer which tends to $0$ by \eqref{eq:asymp}, so $k_n=xu_n$ for all large $n$, and \eqref{eq:pell} gives $h=0$. Therefore $x^2=a=8b+1$ is odd, say $x=2c+1$, and $b=\frac{x^2-1}{8}=T_c$. Finally, $b\ge1$ forces $c\ge1$.
\end{proof}

The proof uses the hypothesis only for large $n$. A stronger form of Theorem~\ref{thm:main}, valid for real $a$ and $b$, is Corollary~\ref{cor:strong}, and Theorem~\ref{thm:trichotomy} shows that much sparser data suffice.

\section{A rigidity principle for quadratic sequences}\label{sec:rigidity}

The mechanism behind Theorem~\ref{thm:main} does not depend on the particular shape of $T_n$. We first record the general fact behind square-root arguments of this kind.

\begin{lemma}\label{lem:squares}
Let $F\in\R[x]$, and suppose that $F(n)=y_n^2$ with $y_n\in\Z$ for all sufficiently large $n\in\N$. Then $F=G^2$ for some $G\in\Q[x]$. If $F\in\Z[x]$, one may take $G\in\Z[x]$.
\end{lemma}

\begin{proof}
We may assume $F\ne0$, so that $F$ has degree $d\ge0$ and positive leading coefficient $\lambda$; replacing $y_n$ by $\lvert y_n\rvert$, we may assume $y_n\ge0$. For large real $x$ we have $\psi(x):=\sqrt{F(x)}=\sqrt\lambda\,x^{d/2}\bigl(1+\epsilon(x)\bigr)^{1/2}$, where $\epsilon(x)$ is a polynomial in $1/x$ without constant term. Hence $\psi(x)=\sqrt\lambda\,x^{d/2}\sum_{j\ge0}\gamma_jx^{-j}$, a convergent expansion that may be differentiated termwise, and $\psi^{(r)}(x)=O(x^{d/2-r})$ for every $r\ge0$. Take $r=\lfloor d/2\rfloor+1>d/2$. By the mean value theorem for finite differences, $\Delta^ry_n=\psi^{(r)}(\xi_n)$ for some $\xi_n\in(n,n+r)$, so $\Delta^ry_n\to0$. Being integers, $\Delta^ry_n=0$ for all $n\ge n_1$, so by Newton's interpolation formula there is $G\in\Q[x]$ of degree less than $r$ with $y_n=G(n)$ for $n\ge n_1$. Then $F(n)=G(n)^2$ for infinitely many $n$, whence $F=G^2$. If $F\in\Z[x]$, write $G=qG_0$ with $q\in\Q_{>0}$ and $G_0\in\Z[x]$ primitive. By Gauss's lemma $G_0^2$ is primitive, so $q^2$, the content of $F$, is an integer, and $q\in\Z$.
\end{proof}

\begin{proposition}\label{prop:rigidity}
Let
\[
f(x)=\lambda(x-\nu)^2+\mu,\qquad g(x)=\lambda'(x-\nu')^2+\mu',
\]
with $\lambda,\lambda'>0$ and $\nu,\nu',\mu,\mu'\in\R$, and let $a>0$ and $b$ be real. Suppose that for every sufficiently large $n\in\N$ there is $m\in\N$ with $af(n)+b=g(m)$. Then there exist $\rho\in\N$ and $\beta\in\Z$ with
\begin{equation}\label{eq:vertex}
a\lambda=\rho^2\lambda',\qquad \beta=\nu'-\rho\nu,\qquad a\mu+b=\mu',
\end{equation}
and consequently $af(x)+b=g(\rho x+\beta)$ identically. Conversely, if $\rho\in\N$, $\beta\in\Z$, $a$ and $b$ satisfy \eqref{eq:vertex}, then $af(n)+b=g(\rho n+\beta)$ for all $n$.
\end{proposition}

\begin{proof}
For large $n$ pick $m_n\in\N$ with $af(n)+b=g(m_n)$. As $g(m_n)\to\infty$, also $m_n\to\infty$. Completing squares,
\[
\lambda'(m_n-\nu')^2=a\lambda(n-\nu)^2+\delta,\qquad \delta:=a\mu+b-\mu'.
\]
With $\rho_0=\sqrt{a\lambda/\lambda'}$ and $\sqrt{1+\varepsilon}=1+O(\varepsilon)$ this gives, for large $n$,
\[
m_n-\nu'=\rho_0(n-\nu)\sqrt{1+\frac{\delta}{a\lambda(n-\nu)^2}}=\rho_0(n-\nu)+O\!\left(\frac1n\right).
\]
Hence $m_{n+1}-m_n=\rho_0+O(1/n)$. Being integers, these differences are eventually equal to $\rho_0$, so $\rho:=\rho_0\in\N$. Then $m_n-\rho n$ is eventually a constant integer $\beta$, and $\beta-\nu'+\rho\nu=(m_n-\nu')-\rho(n-\nu)=O(1/n)$ forces $\beta=\nu'-\rho\nu$. Finally, $\lambda'(m_n-\nu')^2=\lambda'\rho^2(n-\nu)^2=a\lambda(n-\nu)^2$ for large $n$, so $\delta=0$. The identity $af(x)+b=g(\rho x+\beta)$ follows because $\rho x+\beta-\nu'=\rho(x-\nu)$, and the converse is the same computation.
\end{proof}

Applied to $f=g$ with $f(x)=\frac{x(x+1)}{2}=\frac12\bigl(x+\frac12\bigr)^2-\frac18$, Proposition~\ref{prop:rigidity} strengthens Theorem~\ref{thm:main}.

\begin{corollary}\label{cor:strong}
Let $a>0$ and $b\in\R$, and suppose that $aT+b\in\cT$ for all sufficiently large $T\in\cT$. Then $a=(2c+1)^2$ and $b=T_c$ for some integer $c\ge0$, and $aT_n+b=T_{(2c+1)n+c}$ for all $n\in\Z$.
\end{corollary}

\begin{proof}
Here $\lambda=\lambda'=\frac12$, $\nu=\nu'=-\frac12$ and $\mu=\mu'=-\frac18$. Proposition~\ref{prop:rigidity} gives $a=\rho^2$ with $\rho\in\N$ and $\beta=\frac{\rho-1}{2}\in\Z$, so $\rho=2c+1$ is odd, and $b=-\frac18+\frac{\rho^2}{8}=T_c$.
\end{proof}

\subsection{Geometric interpretation}\label{sec:geometry}
Let $\Gamma_f$ and $\Gamma_g$ be the graphs of the polynomials $f$ and $g$ of Proposition~\ref{prop:rigidity}, and for $\rho\ne0$ and real $\beta,a,b$ consider the affine map
\[
\Lambda(x,y)=(\rho x+\beta,\ ay+b)
\]
of the plane. The image $\Lambda(\Gamma_f)$ is the graph of
\[
X\longmapsto a\,f\Bigl(\frac{X-\beta}{\rho}\Bigr)+b=\frac{a\lambda}{\rho^2}\,(X-\rho\nu-\beta)^2+a\mu+b,
\]
a parabola whose vertex is $\Lambda(\nu,\mu)$. Hence $\Lambda(\Gamma_f)=\Gamma_g$ if and only if $a\lambda=\rho^2\lambda'$ and $\Lambda$ maps the vertex $(\nu,\mu)$ of $\Gamma_f$ to the vertex $(\nu',\mu')$ of $\Gamma_g$, and for $\rho>0$ these are exactly the relations \eqref{eq:vertex}. Proposition~\ref{prop:rigidity} can therefore be read as follows: a non-constant affine map $N\mapsto aN+b$ carries $f(\N)$ into $g(\N)$ only if it is the vertical component of an affine map $\Lambda$ with $\Lambda(\Gamma_f)=\Gamma_g$ whose horizontal component $x\mapsto\rho x+\beta$ has $\rho\in\N$ and $\beta\in\Z$. Conversely, such a $\Lambda$ induces a map $f(\N)\to g(\N)$ as soon as, for each $n\in\N$, one of the numbers $\rho n+\beta$ and $2\nu'-\rho n-\beta$ lies in $\N$. In particular, the intercept $b=\mu'-a\mu$ is determined by the slope $a$.

When $f=g$, the map $\Lambda$ fixes the vertex and takes the form
\[
\Lambda(x,y)=\bigl(\nu+\rho(x-\nu),\ \mu+\rho^2(y-\mu)\bigr);
\]
it stretches horizontal distances from the axis of the parabola by the factor $\rho$ and vertical distances from the vertex by the factor $\rho^2$. The induced map on values is the homothety of ratio $\rho^2$ centred at the minimum $\mu$ of $f$ on $\R$. For triangular numbers we have $\nu=-\frac12$ and $\mu=-\frac18$, and the maps of Corollary~\ref{cor:strong} are the \emph{Euler maps}
\[
E_u(N)=u^2N+\frac{u^2-1}{8}\qquad(u\ \text{odd}):
\]
$E_u$ is the homothety of ratio $u^2$ centred at $-\frac18$, and on indices it is the homothety $n\mapsto un+\frac{u-1}{2}$ of ratio $u$ centred at $-\frac12$. Equivalently, $\kappa(N)=8N+1$ maps $\cT$ bijectively onto the odd squares greater than $1$ and satisfies $\kappa\circ E_u=u^2\kappa$. Figure~\ref{fig:euler} illustrates the map $E_3(T)=9T+1$, together with what goes wrong when the vertex is not fixed, and Figure~\ref{fig:pentagonal} shows a map between two different families.

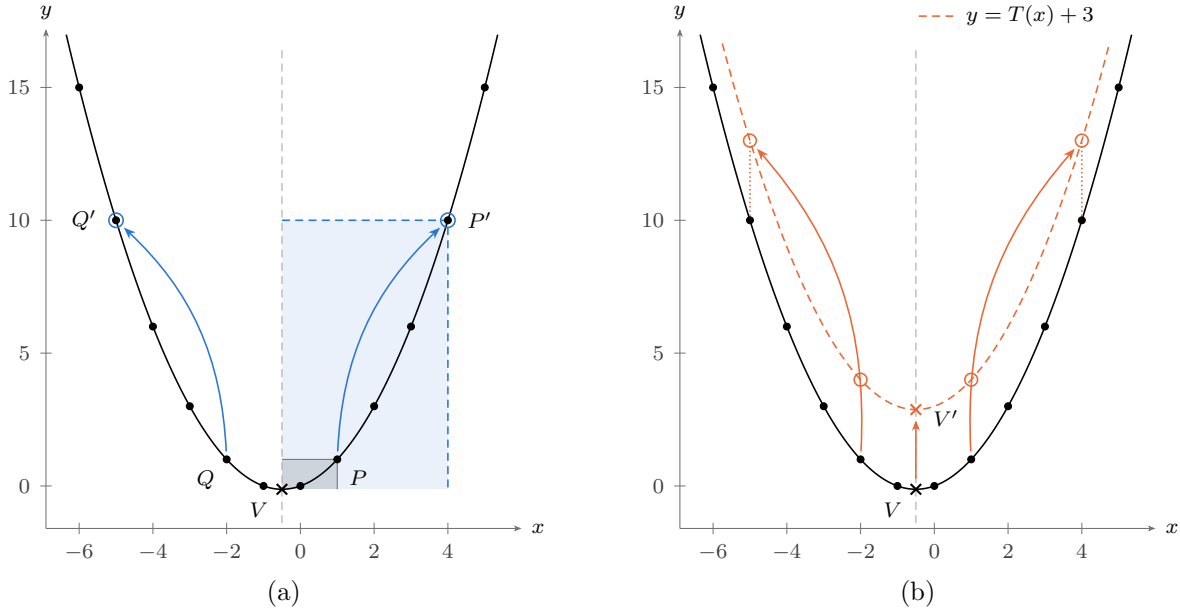
\begin{figure}[t]
\centering
\begin{tikzpicture}
\begin{axis}[parabolaaxis]
  \fill[vizblue, fill opacity=0.10] (axis cs:-0.5,-0.125) rectangle (axis cs:4,10);
  \fill[black, fill opacity=0.12] (axis cs:-0.5,-0.125) rectangle (axis cs:1,1);
  \draw[black!35, densely dashed] (axis cs:-0.5,-1.4) -- (axis cs:-0.5,16.6);
  \draw[black!55] (axis cs:-0.5,1) -- (axis cs:1,1) -- (axis cs:1,-0.125);
  \draw[vizblue, densely dashed, semithick] (axis cs:-0.5,10) -- (axis cs:4,10) -- (axis cs:4,-0.125);
  \addplot[black, semithick, domain=-6.35:5.35, samples=180] {x*(x+1)/2};
  \addplot[latticepts] coordinates {(-6,15) (-5,10) (-4,6) (-3,3) (-2,1) (-1,0) (0,0) (1,1) (2,3) (3,6) (4,10) (5,15)};
  \addplot[only marks, mark=o, mark size=2.7pt, vizblue, semithick] coordinates {(4,10) (-5,10)};
  \draw[-{Stealth[length=4.5pt]}, vizblue, semithick, shorten <=3pt, shorten >=4pt] (axis cs:1,1) to[bend left=22] (axis cs:4,10);
  \draw[-{Stealth[length=4.5pt]}, vizblue, semithick, shorten <=3pt, shorten >=4pt] (axis cs:-2,1) to[bend right=22] (axis cs:-5,10);
  \addplot[only marks, mark=x, mark size=2.8pt, black, thick] coordinates {(-0.5,-0.125)};
  \node[font=\scriptsize, anchor=north east] at (axis cs:-0.6,-0.25) {$V$};
  \node[font=\scriptsize, anchor=north west] at (axis cs:1.05,0.95) {$P$};
  \node[font=\scriptsize, anchor=west] at (axis cs:4.25,10) {$P'$};
  \node[font=\scriptsize, anchor=north east] at (axis cs:-2.05,0.95) {$Q$};
  \node[font=\scriptsize, anchor=east] at (axis cs:-5.25,10) {$Q'$};
  \node[font=\small] at (rel axis cs:0.5,-0.13) {(a)};
\end{axis}
\end{tikzpicture}\hfill
\begin{tikzpicture}
\begin{axis}[parabolaaxis]
  \draw[black!35, densely dashed] (axis cs:-0.5,-1.4) -- (axis cs:-0.5,16.6);
  \addplot[black, semithick, domain=-6.35:5.35, samples=180] {x*(x+1)/2};
  \addplot[vizorange, semithick, densely dashed, domain=-5.75:4.75, samples=180] {x*(x+1)/2+3};
  \addplot[latticepts] coordinates {(-6,15) (-5,10) (-4,6) (-3,3) (-2,1) (-1,0) (0,0) (1,1) (2,3) (3,6) (4,10) (5,15)};
  \draw[vizorange, densely dotted, semithick] (axis cs:4,10) -- (axis cs:4,13);
  \draw[vizorange, densely dotted, semithick] (axis cs:-5,10) -- (axis cs:-5,13);
  \addplot[only marks, mark=o, mark size=2.4pt, vizorange, semithick] coordinates {(1,4) (4,13) (-2,4) (-5,13)};
  \draw[-{Stealth[length=4.5pt]}, vizorange, semithick, shorten <=3pt, shorten >=4pt] (axis cs:1,1) to[bend left=22] (axis cs:4,13);
  \draw[-{Stealth[length=4.5pt]}, vizorange, semithick, shorten <=3pt, shorten >=4pt] (axis cs:-2,1) to[bend right=22] (axis cs:-5,13);
  \draw[-{Stealth[length=4pt]}, vizorange, semithick, shorten <=4pt, shorten >=4pt] (axis cs:-0.5,-0.125) -- (axis cs:-0.5,2.875);
  \addplot[only marks, mark=x, mark size=2.8pt, black, thick] coordinates {(-0.5,-0.125)};
  \addplot[only marks, mark=x, mark size=2.8pt, vizorange, thick] coordinates {(-0.5,2.875)};
  \node[font=\scriptsize, anchor=north east] at (axis cs:-0.6,-0.25) {$V$};
  \node[font=\scriptsize, anchor=west] at (axis cs:-0.3,2.55) {$V'$};
  \node[font=\scriptsize, anchor=south east] (lab) at (axis cs:4.6,16.9) {$y=T(x)+3$};
  \draw[vizorange, semithick, densely dashed] ([xshift=-1pt]lab.west) -- ([xshift=-14pt]lab.west);
  \node[font=\small] at (rel axis cs:0.5,-0.13) {(b)};
\end{axis}
\end{tikzpicture}
\caption{(a)~The Euler map $E_3(T)=9T+1$ is the vertical component of the affine map $(x,y)\mapsto(3x+1,\,9y+1)$, which maps the parabola $y=T(x)=\frac{x(x+1)}{2}$ onto itself and fixes its vertex $V=(-\frac12,-\frac18)$; the dashed vertical line is the axis $x=-\frac12$. The map sends the lattice points $P=(1,1)$ and $Q=(-2,1)$ to $P'=(4,10)$ and $Q'=(-5,10)$, and the rectangle with opposite corners $V$ and $P$ (dark shading) to the rectangle with opposite corners $V$ and $P'$ (light shading), which is three times as wide and nine times as tall. (b)~With the same slope but the intercept $4$, the map $(x,y)\mapsto(3x+1,\,9y+4)$ moves the vertex to $V'=(-\frac12,\frac{23}{8})$ and maps the parabola onto its translate $y=T(x)+3$ (dashed). The images of the lattice points (circles) lie $3$ units above lattice points, and indeed $9T+4$ is never triangular.}
\label{fig:euler}
\end{figure}

\begin{figure}[t]
\centering
\begin{tikzpicture}
\begin{axis}[parabolaaxis, width=0.66\textwidth, height=0.54\textwidth,
    xmin=-6.0, xmax=6.6, xtick={-4,-2,0,2,4,6}]
  \addplot[black, semithick, domain=-5.9:5.28, samples=180] {x*(x+1)/2};
  \addplot[vizblue, semithick, domain=-3.12:3.46, samples=180] {(3*x^2-x)/2};
  \addplot[latticepts] coordinates {(-5,10) (-4,6) (-3,3) (-2,1) (-1,0) (0,0) (1,1) (2,3) (3,6) (4,10) (5,15)};
  \addplot[only marks, mark=square*, mark size=1.5pt, vizblue] coordinates {(-2,7) (-1,2) (0,0) (1,1) (2,5) (3,12)};
  \draw[-{Stealth[length=4.5pt]}, vizgray, semithick, shorten <=3pt, shorten >=3.5pt] (axis cs:-1,2) to[bend right=15] (axis cs:-4,6);
  \draw[-{Stealth[length=4.5pt]}, vizgray, semithick, shorten <=3pt, shorten >=3.5pt] (axis cs:1,1) to[bend right=25] (axis cs:2,3);
  \draw[-{Stealth[length=4.5pt]}, vizgray, semithick, shorten <=3pt, shorten >=3.5pt] (axis cs:2,5) -- (axis cs:5,15);
  \addplot[only marks, mark=x, mark size=2.8pt, black, thick] coordinates {(-0.5,-0.125)};
  \addplot[only marks, mark=x, mark size=2.8pt, vizblue, thick] coordinates {(0.1667,-0.0417)};
  \node[font=\scriptsize, anchor=north east] at (axis cs:-0.6,-0.3) {$V_3$};
  \node[font=\scriptsize, anchor=north west] at (axis cs:0.3,-0.3) {$V_5$};
  \node[font=\scriptsize, anchor=east] at (axis cs:-3.25,16.1) {$y=P_5(x)$};
  \node[font=\scriptsize, anchor=west] at (axis cs:5.35,16.4) {$y=T(x)$};
\end{axis}
\end{tikzpicture}
\caption{The affine map $(x,y)\mapsto(3x-1,\,3y)$ carries the parabola $y=P_5(x)=\frac{x(3x-1)}{2}$ (squares mark its lattice points) onto the parabola $y=T(x)$ (dots), and its vertex $V_5=(\frac16,-\frac1{24})$ onto $V_3=(-\frac12,-\frac18)$. It sends each lattice point $(n,P_5(n))$ to $(3n-1,T_{3n-1})$, so $3P_5(n)=T_{3n-1}$ for all $n$; the arrows join the points with $n=-1,1,2$ to their images, and $(0,0)$ is sent to $(-1,0)$.}
\label{fig:pentagonal}
\end{figure}
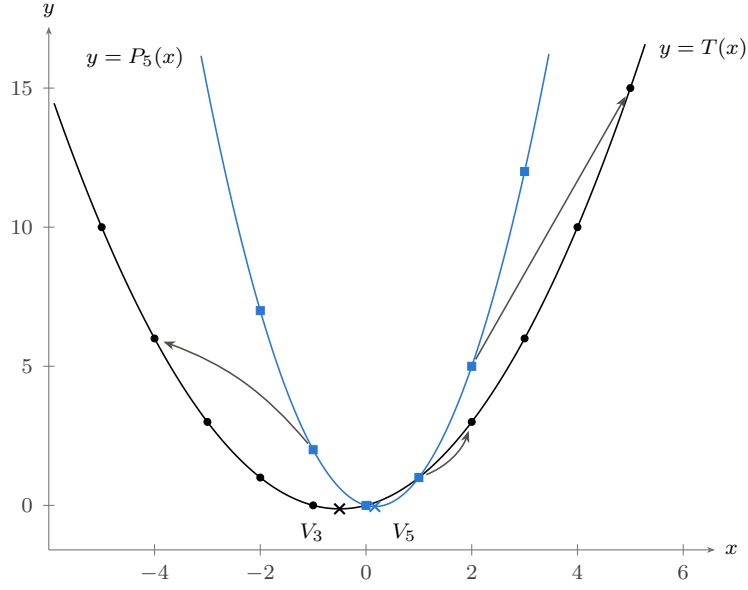

\section{Polygonal numbers}\label{sec:polygonal}

Recall that $P_s(x)=\frac{(s-2)x^2-(s-4)x}{2}$ and $\cP_s=P_s(\N)$ for $s\ge3$; thus $\cP_3=\cT$, $\cP_4$ is the set of squares, $\cP_5$ the set of pentagonal numbers, and so on. Diophantus' identity
\begin{equation}\label{eq:diophantus}
8(s-2)\,P_s(x)+(s-4)^2=\bigl(2(s-2)x-(s-4)\bigr)^2
\end{equation}
is the vertex form $P_s(x)=\lambda_s(x-\nu_s)^2+\mu_s$, with
\[
\lambda_s=\frac{s-2}{2},\qquad \nu_s=\frac{s-4}{2(s-2)}=\frac12-\frac1{s-2},\qquad \mu_s=-\frac{(s-4)^2}{8(s-2)}.
\]
Since $\gcd\bigl(s-4,2(s-2)\bigr)=\gcd(s-4,4)=\gcd(s,4)$, the reduced denominator of $\nu_s$ is
\[
M_s=\frac{2(s-2)}{\gcd(s,4)}=
\begin{cases}
2(s-2), & s\ \text{odd},\\
s-2, & s\equiv2\pmod4,\\
(s-2)/2, & s\equiv0\pmod4,
\end{cases}
\qquad\text{and}\qquad \nu_s=\frac{p_s}{M_s}\ \ \text{with}\ \ p_s=\frac{s-4}{\gcd(s,4)}.
\]
For $s=3,4,\dots,12$ one gets $M_s=2,1,6,4,10,3,14,8,18,5$. For $x\in\N$ put
\[
b_s(x)=\frac{(s-4)^2(x^2-1)}{8(s-2)}=(1-x^2)\mu_s,\qquad E_{s,x}(N)=x^2N+b_s(x),
\]
so that $E_{3,u}=E_u$ is the Euler map of Section~\ref{sec:geometry}.

\begin{theorem}\label{thm:polygonal}
Let $s\ge3$, $a>0$ and $b\in\R$, and suppose that $aP+b\in\cP_s$ for all sufficiently large $P\in\cP_s$. Then $a=x^2$ for some $x\in\N$ with $x\equiv1\pmod{M_s}$, and $b=b_s(x)$. Conversely, let $x\in\N$ with $x\equiv1\pmod{M_s}$, and write $x=1+jM_s$ with $j\ge0$. Then $b_s(x)=P_s(-jp_s)$,
\begin{equation}\label{eq:polyidentity}
x^2P_s(n)+b_s(x)=P_s(xn-jp_s)\qquad(n\in\Z),
\end{equation}
and $xn-jp_s\in\N$ for $n\in\N$; in particular $E_{s,x}(\cP_s)\subseteq\cP_s$.
\end{theorem}

\begin{proof}
Apply Proposition~\ref{prop:rigidity} with $f=g=P_s$. It gives $a=\rho^2$ with $\rho\in\N$, $\beta=(1-\rho)\nu_s\in\Z$, i.e.\ $M_s\mid\rho-1$, and $b=(1-a)\mu_s=b_s(\rho)$. Conversely, if $x=1+jM_s$ then $\beta=-jM_s\nu_s=-jp_s\in\Z$, and Proposition~\ref{prop:rigidity} gives $x^2P_s(n)+b_s(x)=P_s(xn+\beta)$ for all $n$; at $n=0$ this reads $b_s(x)=P_s(-jp_s)$, which yields \eqref{eq:polyidentity}. For $n\ge1$ the index satisfies $xn+\beta\ge x+\beta=1+(x-1)(1-\nu_s)\ge1$ because $\nu_s<1$.
\end{proof}

Theorem~\ref{thmA} is the case of Theorem~\ref{thm:polygonal} in which the hypothesis holds for all $P\in\cP_s$. Call $a\in\N$ \emph{re-$s$-gonal} if $a\cP_s+b\subseteq\cP_s$ for some $b\in\N$.

\begin{corollary}\label{cor:resgonal}
For $s\ne4$, the re-$s$-gonal numbers are the squares $x^2$ with $x>1$ and $x\equiv1\pmod{M_s}$, and the corresponding $b$ is unique. There are no re-square numbers.
\end{corollary}

\begin{proof}
For $s\ne4$ we have $\mu_s<0$, so $b_s(x)=(x^2-1)\lvert\mu_s\rvert$ is positive exactly when $x>1$. For $s=4$ we have $\mu_4=0$, so $b=0$ is forced.
\end{proof}

\begin{table}[ht]
\centering
\renewcommand{\arraystretch}{1.15}
\begin{tabular}{@{}cccl@{}}
\toprule
$s$ & $M_s$ & re-$s$-gonal numbers $a=x^2$ & smallest identity\\
\midrule
3 & 2 & $x$ odd: $9,25,49,\dots$ & $9T_n+1=T_{3n+1}$\\
4 & 1 & none ($b=0$ is forced) & $4P_4(n)=P_4(2n)$\\
5 & 6 & $x\equiv1\ (6)$: $49,169,361,\dots$ & $49P_5(n)+2=P_5(7n-1)$\\
6 & 4 & $x\equiv1\ (4)$: $25,81,169,\dots$ & $25P_6(n)+3=P_6(5n-1)$\\
7 & 10 & $x\equiv1\ (10)$: $121,441,\dots$ & $121P_7(n)+27=P_7(11n-3)$\\
8 & 3 & $x\equiv1\ (3)$: $16,49,100,\dots$ & $16P_8(n)+5=P_8(4n-1)$\\
9 & 14 & $x\equiv1\ (14)$: $225,841,\dots$ & $225P_9(n)+100=P_9(15n-5)$\\
10 & 8 & $x\equiv1\ (8)$: $81,289,\dots$ & $81P_{10}(n)+45=P_{10}(9n-3)$\\
11 & 18 & $x\equiv1\ (18)$: $361,1369,\dots$ & $361P_{11}(n)+245=P_{11}(19n-7)$\\
12 & 5 & $x\equiv1\ (5)$: $36,121,256,\dots$ & $36P_{12}(n)+28=P_{12}(6n-2)$\\
\bottomrule
\end{tabular}
\caption{Re-$s$-gonal numbers and the smallest nontrivial self-map of $\cP_s$ (Theorem~\ref{thm:polygonal}).}
\label{tab:self}
\end{table}

\begin{remark}\label{rem:poly}
Even squares occur exactly when $M_s$ is odd, that is, when $s\equiv0\pmod4$; for example $16P_8(n)+5=P_8(4n-1)$. Since the index maps $n\mapsto xn-jp_s$ compose, the self-maps satisfy $E_{s,x}\circ E_{s,y}=E_{s,xy}$; the resulting monoid is studied in Section~\ref{sec:general}.
\end{remark}

The same method applies to the \emph{generalized} $s$-gonal numbers $P_s(\Z)$, for instance the generalized pentagonal numbers $0,1,2,5,7,12,15,\dots$ of Euler's pentagonal number theorem.

\begin{proposition}\label{prop:generalized}
Let $s\ge3$, $a>0$ and $b\in\R$. Then $aP_s(\Z)+b\subseteq P_s(\Z)$ if and only if $a=x^2$ for some $x\in\N$ with $x\equiv\pm1\pmod{M_s}$, and $b=b_s(x)$.
\end{proposition}

\begin{proof}
Since $P_s(0)=0$ and $P_s(1)=1$, both $b$ and $a+b$ lie in $P_s(\Z)\subseteq\Z$, so $a\in\Z$. Put $w_n=2(s-2)n-(s-4)$ and $F(n)=8(s-2)\bigl(aP_s(n)+b\bigr)+(s-4)^2$. By \eqref{eq:diophantus}, $F(n)$ is a perfect square for every $n\in\Z$, and $F(n)=aw_n^2+\delta$ with $\delta=8(s-2)b-(a-1)(s-4)^2$. By Lemma~\ref{lem:squares}, $F=G^2$ for some $G\in\Q[x]$, necessarily of the form $G(n)=\alpha w_n+\gamma$; comparing coefficients gives $\alpha^2=a$, $\gamma=0$ and $\delta=0$. Hence $\sqrt a\in\Q$, so $a=x^2$ with $x\in\N$, and $b=b_s(x)$. Then $aP_s(n)+b=\lambda_sx^2(n-\nu_s)^2+\mu_s$, and this equals $P_s(m)$ exactly when $m-\nu_s=\pm x(n-\nu_s)$, that is, $m\in\{xn-(x-1)\nu_s,\ -xn+(x+1)\nu_s\}$. For $n=0$ one of these numbers must be an integer, so $(x-1)\nu_s\in\Z$ or $(x+1)\nu_s\in\Z$, i.e.\ $x\equiv1$ or $x\equiv-1\pmod{M_s}$. Conversely, either condition makes the corresponding value of $m$ an integer for every $n\in\Z$.
\end{proof}

For the generalized pentagonal numbers, $24G+1=(6n-1)^2$ and $M_5=6$, so Proposition~\ref{prop:generalized} gives the maps $G\mapsto x^2G+\frac{x^2-1}{24}$ for every $x$ prime to $6$, such as $25G+1$ and $49G+2$; only those with $x\equiv1\pmod6$ preserve the ordinary pentagonal numbers.

The self-maps of Theorem~\ref{thm:polygonal} also \emph{reflect} membership in $\cP_s$.

\begin{proposition}\label{prop:reflection}
Let $s\ge3$ and $x\in\N$ with $x\equiv1\pmod{M_s}$. For $N\in\N$ we have $N\in\cP_s$ if and only if $E_{s,x}(N)\in\cP_s$. In particular, $N$ is triangular if and only if $9N+1$ is triangular.
\end{proposition}

\begin{proof}
One direction is Theorem~\ref{thm:polygonal}. Conversely, suppose that $x^2N+b_s(x)=P_s(m)$ with $m\in\N$. Put $D=2(s-2)$, $r=s-4$ and $w=Dm-r$, so that $w\ge D-r=s>0$. Since $8(s-2)b_s(x)=r^2(x^2-1)$, identity \eqref{eq:diophantus} gives
\[
w^2=8(s-2)\bigl(x^2N+b_s(x)\bigr)+r^2=x^2\bigl(8(s-2)N+r^2\bigr).
\]
Hence $8(s-2)N+r^2=v^2$ with $v=w/x$; as $v$ is a positive rational number whose square is an integer, $v\in\N$. We claim that $\gcd(s,4)\mid v$. This is clear if $s$ is odd. If $s\equiv2\pmod4$, then $r$ is even, so $v^2$ and hence $v$ are even. If $s\equiv0\pmod4$, then $4\mid r$ and $16\mid8(s-2)$, so $16\mid v^2$ and $4\mid v$. Writing $x=1+jM_s$ and using $D=\gcd(s,4)M_s$, we obtain $jM_sv\equiv0\pmod D$, hence $v\equiv xv=w\equiv-r\pmod D$. Thus $v=Dn-r$ for some $n\in\Z$, and $n\ge1$ because $v^2>r^2$ forces $v>\lvert r\rvert$. Now \eqref{eq:diophantus} gives $8(s-2)N+r^2=(Dn-r)^2=8(s-2)P_s(n)+r^2$, so $N=P_s(n)\in\cP_s$.
\end{proof}

\section{Maps between polygonal families}\label{sec:between}

\begin{theorem}\label{thm:between}
Let $s,t\ge3$, $a>0$ and $b\in\R$. Then $a\cP_s+b\subseteq\cP_t$ if and only if there is $\rho\in\N$ such that
\[
\beta:=\nu_t-\rho\,\nu_s\in\Z,\qquad a=\rho^2\,\frac{t-2}{s-2},\qquad b=P_t(\beta).
\]
In that case $aP_s(n)+b=P_t(\rho n+\beta)$ for all $n\in\Z$, and automatically $a\in\N$ and $b\in\Z_{\ge0}$. Moreover:
\begin{enumerate}[label=(\roman*),leftmargin=*]
\item an admissible $\rho$ exists if and only if $M_t\mid M_s$, and then the admissible values of $\rho$ are exactly the positive integers in a single residue class modulo $M_s$;
\item composing maps $\cP_s\to\cP_t\to\cP_u$ multiplies the parameters $\rho$.
\end{enumerate}
\end{theorem}

\begin{proof}
Proposition~\ref{prop:rigidity} with $f=P_s$ and $g=P_t$ gives $a\lambda_s=\rho^2\lambda_t$, i.e.\ $a=\rho^2\frac{t-2}{s-2}$, together with $\beta=\nu_t-\rho\nu_s\in\Z$ and the identity $aP_s(x)+b=P_t(\rho x+\beta)$. Evaluating at $x=0$ and $x=1$ gives $b=P_t(\beta)$ and $a=P_t(\rho+\beta)-P_t(\beta)$, both integers. The roots of $P_t$ are $0$ and $2\nu_t\in[-1,1)$, and no integer lies strictly between them, so $P_t(\beta)\ge0$. Conversely, given such $\rho$, the relations \eqref{eq:vertex} hold for $f=P_s$ and $g=P_t$: indeed $a\lambda_s=\rho^2\lambda_t$, and since $\lambda_s\nu_s^2+\mu_s=P_s(0)=0$, we get $b=P_t(\beta)=\lambda_t\rho^2\nu_s^2+\mu_t=\mu_t-a\mu_s$. Hence $aP_s(n)+b=P_t(\rho n+\beta)$ for all $n$ by Proposition~\ref{prop:rigidity}, and for $n\ge1$ the index satisfies $\rho n+\beta\ge\rho(1-\nu_s)+\nu_t>\frac{\rho}{2}-\frac12\ge0$, because $1-\nu_s>\frac12$ and $\nu_t\ge-\frac12$; so $\rho n+\beta\in\N$.

(i) The condition $\beta\in\Z$ says that $\rho\nu_s\equiv\nu_t\pmod 1$. Since $\gcd(p_s,M_s)=1$, the map $\rho\mapsto\rho\nu_s$ induces a bijection $\Z/M_s\Z\to\frac1{M_s}\Z/\Z$. Hence a solution exists if and only if $\nu_t\in\frac1{M_s}\Z$, that is, $M_t\mid M_s$, and the solutions form one residue class modulo $M_s$. (ii) follows from $\rho'(\rho n+\beta)+\beta'=\rho'\rho\,n+(\rho'\beta+\beta')$.
\end{proof}

\begin{corollary}\label{cor:divisibility}
The map $s\mapsto M_s$ is a bijection from $\{3,4,5,\dots\}$ onto $\N$, with inverse
\[
M\longmapsto
\begin{cases}
2M+2, & M\ \text{odd},\\
M/2+2, & M\equiv2\pmod4,\\
M+2, & M\equiv0\pmod4,
\end{cases}
\]
and there is a non-constant affine map sending $\cP_s$ into $\cP_t$ if and only if $M_t\mid M_s$. In particular:
\begin{enumerate}[label=(\alph*),leftmargin=*]
\item every family maps into the squares, and the map with the smallest $\rho$, namely $\rho=M_s$, is $\frac{2M_s^2}{s-2}\,P_s(n)+p_s^2=(M_sn-p_s)^2$; for odd $s$ this is Diophantus' identity \eqref{eq:diophantus};
\item the squares map only into the squares;
\item $\cP_s$ maps into $\cT$ if and only if $M_s$ is even, that is, $s\not\equiv0\pmod4$;
\item $\cP_s$ and $\cP_t$ map into each other only when $s=t$.
\end{enumerate}
\end{corollary}

\begin{proof}
If $s$ is odd, then $M_s=2(s-2)\equiv2\pmod4$; if $s\equiv2\pmod4$, then $M_s=s-2\equiv0\pmod4$; and if $s\equiv0\pmod4$, then $M_s=\frac{s-2}{2}$ is odd. In each of the three cases $s$ is recovered from $M_s$ by the stated formula, and every $M$ in the corresponding class arises, so $s\mapsto M_s$ is bijective. The second statement is Theorem~\ref{thm:between}(i), because a non-constant affine map sending $\cP_s$ into $\cP_t$ has positive slope ($\cP_s$ is unbounded above and $\cP_t$ is bounded below). For (a), $M_4=1$, and for $t=4$ the condition $\beta=-\rho\nu_s\in\Z$ means $M_s\mid\rho$; the formula follows from $a=\frac{2\rho^2}{s-2}$ and $b=P_4(-p_s)=p_s^2$. Items (b)--(d) follow from $M_t\mid1\iff t=4$, from $M_3=2$, and from the injectivity of $s\mapsto M_s$.
\end{proof}

\begin{corollary}\label{cor:inclusion}
The only inclusion $\cP_s\subsetneq\cP_t$ between polygonal families is $\cP_6\subset\cP_3$: hexagonal numbers are triangular, $P_6(n)=T_{2n-1}$.
\end{corollary}

\begin{proof}
An inclusion is the case $a=1$, $b=0$ of Theorem~\ref{thm:between}: then $\rho^2(t-2)=s-2$ and $P_t(\beta)=0$, i.e.\ $\beta\in\{0,2\nu_t\}$. If $\beta=0$, then $\rho\nu_s=\nu_t$, i.e.\ $\rho\bigl(\frac12-\frac1{\rho^2(t-2)}\bigr)=\frac12-\frac1{t-2}$, which simplifies to $(\rho-1)\bigl(\rho(t-2)+2\bigr)=0$; so $\rho=1$ and $s=t$. If $\beta=2\nu_t\ne0$ is an integer, then $t=3$ and $\beta=-1$, so $\rho\nu_s=\frac12$ with $s-2=\rho^2$; this reads $\frac\rho2-\frac1\rho=\frac12$, whence $\rho=2$ and $s=6$.
\end{proof}

\begin{table}[ht]
\centering
\renewcommand{\arraystretch}{1.15}
\begin{tabular}{@{}lccl@{}}
\toprule
identity & $s\to t$ & $\rho$ & comment\\
\midrule
$8T_n+1=(2n+1)^2$ & $3\to4$ & $2$ & Plutarch\\
$P_6(n)=T_{2n-1}$ & $6\to3$ & $2$ & hexagonal $\subset$ triangular\\
$3P_5(n)=T_{3n-1}$ & $5\to3$ & $3$ & Figure~\ref{fig:pentagonal}\\
$24P_5(n)+1=(6n-1)^2$ & $5\to4$ & $6$ & Diophantus, $s=5$\\
$8P_5(n)=P_8(2n)$ & $5\to8$ & $2$ & \\
$3P_8(n)+1=(3n-1)^2$ & $8\to4$ & $3$ & \\
$5P_7(n)+1=T_{5n-2}$ & $7\to3$ & $5$ & \\
$128P_7(n)+28=P_{12}(8n-2)$ & $7\to12$ & $8$ & target with larger $s$\\
$25P_{14}(n)+26=P_5(10n-4)$ & $14\to5$ & $10$ & \\
\bottomrule
\end{tabular}
\caption{Some affine maps between polygonal families (Theorem~\ref{thm:between}); in each case $\rho$ is the smallest admissible value.}
\label{tab:between}
\end{table}

\begin{table}[ht]
\centering
\renewcommand{\arraystretch}{1.1}
\begin{tabular}{@{}ccl@{\qquad}ccl@{}}
\toprule
$s$ & $M_s$ & targets $t$ & $s$ & $M_s$ & targets $t$\\
\midrule
3 & 2 & 3, 4 & 9 & 14 & 3, 4, 9, 16\\
4 & 1 & 4 & 10 & 8 & 3, 4, 6, 10\\
5 & 6 & 3, 4, 5, 8 & 11 & 18 & 3, 4, 5, 8, 11, 20\\
6 & 4 & 3, 4, 6 & 12 & 5 & 4, 12\\
7 & 10 & 3, 4, 7, 12 & 13 & 22 & 3, 4, 13, 24\\
8 & 3 & 4, 8 & 14 & 12 & 3, 4, 5, 6, 8, 14\\
\bottomrule
\end{tabular}
\caption{All $t$ such that $\cP_s$ maps affinely into $\cP_t$ (Corollary~\ref{cor:divisibility}).}
\label{tab:targets}
\end{table}

\section{Complements}\label{sec:further}

\subsection{Sparse data suffice}\label{sec:sparse}
For $a,b\in\N$ let $\cN(a,b)=\{n\in\N:\ aT_n+b\in\cT\}$.

\begin{theorem}\label{thm:trichotomy}
Let $a,b\in\N$. Exactly one of the following holds:
\begin{enumerate}[label=(\roman*),leftmargin=*]
\item $a=(2c+1)^2$ and $b=T_c$, and then $\cN(a,b)=\N$;
\item $a$ is a square but \textup{(i)} fails, and then $\cN(a,b)$ is finite;
\item $a$ is not a square, and then $\lvert\cN(a,b)\cap[1,N]\rvert=O_{a,b}(\log N)$; moreover, $\cN(a,b)$ is either empty or infinite.
\end{enumerate}
In particular, if $\limsup_{N\to\infty}\lvert\cN(a,b)\cap[1,N]\rvert/\log N=\infty$ (for instance, if $\cN(a,b)$ has positive upper density or contains an infinite arithmetic progression), then $a=(2c+1)^2$ and $b=T_c$.
\end{theorem}

\begin{proof}
Put $h=8b+1-a$. Then $n\in\cN(a,b)$ if and only if $a(2n+1)^2+h$ is a perfect square; its square root is then automatically odd and at least $3$. If $a=x^2$ and $h\ne0$, the equation $(k-xu)(k+xu)=h$ has only finitely many solutions. If $a=x^2$ and $h=0$, then $x^2=8b+1$ and we are in case (i). If $a$ is not a square, then $h\ne0$ unless $\cN(a,b)=\emptyset$, and we are counting solutions of the generalized Pell equation $k^2-au^2=h$ with $u=2n+1$ \cite{NZM}. Its solutions with $k,u>0$ fall into finitely many classes $\{(k_i+u_i\sqrt a)\,\varepsilon^j:\ j\in\Z\}$, where $\varepsilon>1$ generates the group of units of norm $1$ of $\Z[\sqrt a]$ modulo $\pm1$. Within each class the numbers $k+u\sqrt a$ grow geometrically with ratio $\varepsilon$, so only $O(\log N)$ of them have $u\le 2N+1$. Finally, suppose that $(k_0,u_0)$ is a solution with $u_0=2n_0+1$. Choose $r\ge1$ with $\varepsilon^r\equiv1\pmod{2\Z[\sqrt a]}$, which is possible because $\varepsilon$ is a unit of the finite ring $\Z[\sqrt a]/2\Z[\sqrt a]$. Multiplying $k_0+u_0\sqrt a$ by $\varepsilon^{rj}$, $j=1,2,\dots$, preserves the parities of $k$ and $u$ and produces infinitely many elements of $\cN(a,b)$.
\end{proof}

For example, the sets
\[
\cN(2,1)=\{1,4,9,26,55,154,323,900,\dots\},\qquad \cN(3,1)=\{2,9,35,132,494,\dots\}
\]
grow geometrically, while $\cN(4,3)=\{2\}$, since $4\cdot3+3=15$, and $\cN(5,2)=\emptyset$.

\begin{remark}\label{rem:sparse-general}
The same method applies to any two families. Let $s,t\ge3$, let $a,b\in\Q$ with $a>0$, and let $\cN=\{n\in\N:\ aP_s(n)+b\in\cP_t\}$. If $aP_s(n)+b=P_t(m)$, then by \eqref{eq:diophantus} the numbers $v=2(t-2)m-(t-4)$ and $w=2(s-2)n-(s-4)$ satisfy $v^2=\theta^2w^2+C$ with $\theta^2=\frac{a(t-2)}{s-2}$ and a constant $C$ depending only on $a,b,s,t$. If $\theta^2$ is not the square of a rational number, the Pell argument above gives $\lvert\cN\cap[1,N]\rvert=O(\log N)$; if $\theta\in\Q$ and $C\ne0$, then $\cN$ is finite; and if $\theta\in\Q$ and $C=0$, then $v=\theta w$ is an affine function of $n$ with rational coefficients, and $n\in\cN$ if and only if the corresponding number $m=\frac{v+t-4}{2(t-2)}$ is a positive integer; apart from finitely many $n$, this means that $n$ lies in a certain union of residue classes. Such residue classes do occur when $s\ne t$: $T_n\in\cP_6$ if and only if $n$ is odd.
\end{remark}

\subsection{Polynomial maps}\label{sec:polymaps}
Let $\iota(x)=-1-x$, so that $T_{\iota(x)}=T_x$.

\begin{theorem}\label{thm:polymaps}
\begin{enumerate}[label=(\alph*),leftmargin=*]
\item For every $g\in\Q[x]$ with $g\circ\iota=g$ or $g\circ\iota=\iota\circ g$ there is a unique $P_g\in\Q[x]$ with $P_g(T_x)=T_{g(x)}$, and $P_g(\cT)\subseteq\cT$ whenever $g(\N)\subseteq\Z\setminus\{0,-1\}$.
\item Conversely, every $P\in\Q[x]$ with $P(\cT)\subseteq\cT$ is equal to $P_g$ for such a $g$, which is unique up to replacing $g$ by $\iota\circ g$.
\item $P_g\circ P_{g'}=P_{g\circ g'}$.
\end{enumerate}
\end{theorem}

\begin{proof}
(a) In both cases $T_{g(x)}$ is invariant under $\iota$. In the variable $y=x+\frac12$, $\iota$ is $y\mapsto-y$ and $y^2=2T_x+\frac14$, so the $\iota$-invariant polynomials are exactly the elements of $\Q[T_x]$. This gives $P_g$, and uniqueness holds because $T_x$ is not constant. If $g(n)\in\Z\setminus\{0,-1\}$, then $P_g(T_n)=T_{g(n)}\in\cT$.

(b) Let $F(x)=8P(T_x)+1$. By \eqref{eq:plutarch}, $F(n)$ is the square of an odd integer for every $n\in\N$, so Lemma~\ref{lem:squares} gives $F=G^2$ with $G\in\Q[x]$, and $G(n)$ is an odd integer for every $n\in\N$. Put $g=(G-1)/2$; then $T_{g(x)}=\frac{G(x)^2-1}{8}=P(T_x)$ and $g(\N)\subseteq\Z$. Since $F\circ\iota=F$, we get $(G\circ\iota-G)(G\circ\iota+G)=0$ in $\Q[x]$, so $G\circ\iota=\pm G$; this translates into $g\circ\iota=g$ or $g\circ\iota=-1-g=\iota\circ g$. Moreover $g(n)\notin\{0,-1\}$ for $n\in\N$, because $T_{g(n)}=P(T_n)\ge1$. If also $T_{g}=T_{g'}$, then $(g-g')(g+g'+1)=0$, so $g'\in\{g,\iota\circ g\}$.

(c) Substitute $x\mapsto g'(x)$ in $P_g(T_x)=T_{g(x)}$, and note that $g\circ g'$ again satisfies the symmetry condition.
\end{proof}

\begin{example}\label{ex:polymaps}
The polynomials $g(x)=(2c+1)x+c$ give the Euler maps. The polynomial $g(x)=T_x$ gives $P(t)=\frac{t(t+1)}{2}$, so $T_{T_n}$ is triangular; $g(x)=2x(x+1)$ gives $P(t)=T_{4t}=8t^2+2t$; and $g(x)=\frac{(2x+1)^3-1}{2}$ gives $P(t)=\frac{(8t+1)^3-1}{8}=64t^3+24t^2+3t$. In terms of $G=2g+1$ and $u=2x+1$, Theorem~\ref{thm:polymaps} says that $G$ is an even or an odd polynomial in $u$, i.e.\ $8P(t)+1=H(8t+1)^2$ or $8P(t)+1=(8t+1)H(8t+1)^2$ with $H\in\Q[z]$. In the even case $P(t)=T_{h(t)}$ with $h(t)=\frac{H(8t+1)-1}{2}$, so $P$ factors through $y\mapsto T_y$; the odd case is the instance $n=2$, $r=1$ (with the same $H$) of the family $x^rH(x)^n\circ x^n=x^n\circ x^rH(x^n)$ in Ritt's second theorem \cite{Ritt}.
\end{example}

\subsection{Higher degree}\label{sec:higher}

\begin{theorem}\label{thm:higher}
Let $f,g\in\R[x]$ have the same degree $d\ge2$ and positive leading coefficients $\lambda_f,\lambda_g$, and write $f(x)=\lambda_f(x+\theta_f)^d+O(x^{d-2})$ and $g(x)=\lambda_g(x+\theta_g)^d+O(x^{d-2})$. Let $a>0$ and $b\in\R$, and suppose that $af(n)+b\in g(\N)$ for all sufficiently large $n\in\N$. Then $\rho:=(a\lambda_f/\lambda_g)^{1/d}\in\N$, $\beta:=\rho\theta_f-\theta_g\in\Z$, and $af(x)+b=g(\rho x+\beta)$ identically. If moreover $f=g$ and $a\ne1$, then $a=\rho^d$ with $\rho\ge2$, and
\[
f(x)=\lambda(x-\nu)^d+\mu\quad\text{for some }\lambda,\nu,\mu\in\R,\qquad\text{with }(\rho-1)\nu\in\Z\ \text{ and }\ b=(1-a)\mu.
\]
\end{theorem}

\begin{proof}
Write $af(n)+b=g(m_n)$ with $m_n\in\N$; then $m_n\to\infty$. Inverting $g$ near $+\infty$ gives $m=(Y/\lambda_g)^{1/d}-\theta_g+O(Y^{-1/d})$ whenever $g(m)=Y$, while $\bigl((af(n)+b)/\lambda_g\bigr)^{1/d}=\rho(n+\theta_f)+O(1/n)$. Hence $m_n=\rho n+\rho\theta_f-\theta_g+O(1/n)$, and, exactly as in the proof of Proposition~\ref{prop:rigidity}, $m_{n+1}-m_n$ is eventually equal to $\rho\in\N$ and $m_n-\rho n$ is eventually equal to $\beta=\rho\theta_f-\theta_g\in\Z$. Thus $af(n)+b=g(\rho n+\beta)$ for all large $n$, and hence identically.

If $f=g$ and $a\ne1$, then $\rho\ne1$, so $\rho\ge2$, and $L(x)=\rho x+\beta$ has the fixed point $\nu=-\beta/(\rho-1)$. Put $F(y)=f(\nu+y)=\sum_{j=0}^dF_jy^j$. From $f\circ L=af+b$ we get $F(\rho y)=\rho^dF(y)+b$, that is, $\rho^jF_j=\rho^dF_j$ for $1\le j\le d$ and $F_0=\rho^dF_0+b$. Hence $F_j=0$ for $1\le j<d$, so $f(x)=F_d(x-\nu)^d+F_0$ and $b=(1-\rho^d)F_0$; also $(\rho-1)\nu=-\beta\in\Z$.
\end{proof}

Thus, in degree $d\ge3$, a set $f(\N)$ admits an affine self-map other than the identity only if $f$ is a translated monomial $\lambda(x-\nu)^d+\mu$. For $a=1$ the theorem gives $\rho=1$, $\beta=0$ and $b=0$, so there are no nontrivial translations in any degree $d\ge2$.

\begin{example}\label{ex:cubic}
The tetrahedral numbers $\mathrm{Te}_n=\binom{n+2}{3}=\frac{y^3-y}{6}$ ($y=n+1$) and the square pyramidal numbers $\mathrm{Py}_n=\frac{n(n+1)(2n+1)}{6}=\frac{w^3-w}{24}$ ($w=2n+1$) have a nonzero linear term in the centred variable, so neither set has an affine self-map other than the identity. The same holds for the octahedral numbers $\frac{2n^3+n}{3}$, the centred cube numbers $n^3+(n-1)^3=2y^3+\frac32y$ ($y=n-\frac12$) and the pentatope numbers $\binom{n+3}{4}$, while the cubes admit only the maps $N\mapsto\rho^3N$. Between different families, however, affine maps can exist: Theorem~\ref{thm:higher} shows that $4\,\mathrm{Py}_n=\mathrm{Te}_{2n}$ is the only non-constant affine map from the square pyramidal numbers into the tetrahedral numbers, and that there is none in the opposite direction.
\end{example}

\subsection{Sums of two triangular numbers}\label{sec:two}
Let $\cS=\{T_i+T_j:\ i,j\ge0\}$, where $T_0=0$ is allowed. Writing $8N+2=(2i+1)^2+(2j+1)^2$ and using $2(x^2+y^2)=(x+y)^2+(x-y)^2$, one sees that $N\in\cS$ if and only if $N\ge0$ and $4N+1$ is a sum of two squares.

\begin{theorem}\label{thm:two}
Let $a\in\N$ and $b\in\Z$. Then $a\cS+b\subseteq\cS$ if and only if $a\equiv1\pmod4$, $a$ is a sum of two squares, and $b=\frac{a-1}{4}$.
\end{theorem}

\begin{proof}
Let $\Sigma$ denote the set of sums of two squares. If $a\in\Sigma$, $a\equiv1\pmod4$ and $b=\frac{a-1}4$, then $4(aN+b)+1=a(4N+1)$ is a product of two elements of $\Sigma$, hence lies in $\Sigma$ by the Brahmagupta--Fibonacci identity; so $aN+b\in\cS$ for every $N\in\cS$.

Conversely, suppose that $a\cS+b\subseteq\cS$, and put $h=4b+1-a$, so that $4(aN+b)+1=a(4N+1)+h$. Assume that $h\ne0$, and choose a prime $\ell\equiv3\pmod4$ with $\ell\nmid ah$. Let $r$ be an integer with $ar+h\equiv \ell\pmod{\ell^2}$; then $\ell\nmid r$. Every residue modulo $\ell$ is a sum of two squares; since $r\not\equiv0\pmod \ell$, a representation $r\equiv x_0^2+y_0^2\pmod \ell$ has, say, $x_0\not\equiv0$, so by Hensel's lemma it lifts to $r\equiv x_1^2+y_0^2\pmod{\ell^2}$. By the Chinese remainder theorem we may pick $x\equiv x_1$ and $y\equiv y_0\pmod{\ell^2}$ with $x$ odd and $y$ even. Then $m=x^2+y^2\equiv1\pmod4$, so $N=\frac{m-1}{4}\in\cS$, while $4(aN+b)+1=am+h\equiv \ell\pmod{\ell^2}$ is divisible by $\ell$ exactly once and therefore is not a sum of two squares. Thus $aN+b\notin\cS$, a contradiction. Hence $h=0$, i.e.\ $a=4b+1$, so $a\equiv1\pmod4$; and $b=a\cdot0+b\in\cS$ gives $a=4b+1\in\Sigma$.
\end{proof}

\begin{remark}\label{rem:two}
For one triangular number the admissible multipliers are the odd squares, that is, the values of $x^2$ that are $\equiv1\pmod8$; for two triangular numbers they are the values of $x^2+y^2$ (the norms from $\Z[i]$) that are $\equiv1\pmod4$, so that, for example, $5N+1$, $13N+3$ and $17N+4$ preserve $\cS$. For an odd square $a=(2c+1)^2$ the map is $N\mapsto aN+2T_c$, and $a(T_i+T_j)+2T_c=E_{2c+1}(T_i)+E_{2c+1}(T_j)$; Euler noted the cases $9N+2$, $25N+6$ and $49N+12$ as corollaries of his observations on $9T+1$, $25T+3$ and $49T+6$ \cite[\S87]{Euler}. For three or more triangular numbers there is no restriction, since by Gauss's theorem every nonnegative integer is a sum of three triangular numbers; for weighted sums of triangular numbers representing every positive integer, see the ``triangular theorem of eight'' of Bosma and Kane \cite{BosmaKane}. In both nontrivial cases the maps are homotheties centred at the minimum of the corresponding quadratic polynomial on $\R^k$, namely $-\frac18$ and $-\frac14$. Note that $\cS$ contains the consecutive integers $T_k$ and $T_k+1$ for every $k$, so its gaps do not tend to infinity and Theorem~\ref{thmC} does not apply; the proof of Theorem~\ref{thm:two} uses a local obstruction instead of growth.
\end{remark}

\section{The monoid of self-maps of a set of integers}\label{sec:general}

For $S\subseteq\R$ let $\End(S)$ be the set of affine maps $\phi(x)=ax+b$ with $\phi(S)\subseteq S$. It contains the identity and the constant maps with values in $S$, and it is closed under composition, so it is a monoid; its non-constant elements form a submonoid $\End^*(S)$, whose elements we call the \emph{self-maps} of $S$. Let $\Aut(S)$ be the set of self-maps $\phi$ with $\phi(S)=S$.

\subsection{Basic structure}

\begin{proposition}\label{prop:basic}
Let $S\subseteq\Z$ with $\lvert S\rvert\ge2$, and let $\phi(x)=ax+b$ be a self-map of $S$.
\begin{enumerate}[label=(\alph*),leftmargin=*]
\item $a$ and $b$ are integers; in particular $\lvert a\rvert\ge1$, and $\phi$ is injective.
\item $\Aut(S)$ is the group of units of $\End(S)$, and it is a subgroup of the infinite dihedral group $D_\infty=\{x\mapsto\pm x+k:\ k\in\Z\}$. Hence $\Aut(S)$ is trivial, or generated by one reflection (and $S$ is symmetric about a point), or generated by one translation $x\mapsto x+d$ (and $S$ is periodic), or generated by one reflection and one translation.
\item If $S$ is finite, then $\End^*(S)=\Aut(S)$, and this group is either $\{\id\}$ or $\{\id,\sigma\}$ with $\sigma(x)=\min S+\max S-x$; the latter happens exactly when $S$ is symmetric about its midpoint.
\item If $S$ is infinite and bounded below, then $a\ge1$ and $\Aut(S)=\{\id\}$; that is, every self-map other than the identity is injective but not surjective.
\end{enumerate}
\end{proposition}

\begin{proof}
(a) Pick $s\ne s'$ in $S$. Then $a=\frac{\phi(s)-\phi(s')}{s-s'}\in\Q$; write $a=p/q$ in lowest terms with $q\ge1$. Since $\phi^k(S)\subseteq S\subseteq\Z$, the number $\phi^k(s)-\phi^k(s')=\frac{p^k}{q^k}(s-s')$ is an integer for every $k$, so $q^k\mid s-s'$ for all $k$, whence $q=1$. Then $b=\phi(s)-as\in\Z$, and $\phi$ is injective because $a\ne0$.

(b) If $\phi(S)=S$, then the affine map $\phi^{-1}$ also satisfies $\phi^{-1}(S)=S$, so $\phi$ is a unit; conversely, if $\phi\circ\psi=\id$ with $\psi\in\End(S)$, then $S=\phi(\psi(S))\subseteq\phi(S)$. By (a) both $a$ and $a^{-1}$ are integers, so $a=\pm1$. The list of cases is the list of subgroups of $D_\infty$.

(c) An injective self-map of a finite set is bijective, so $\End^*(S)=\Aut(S)$. A translation $x\mapsto x+k$ preserving a finite nonempty set has $k=0$, and a reflection $x\mapsto k-x$ preserving $S$ must exchange $\min S$ and $\max S$.

(d) An infinite set of integers that is bounded below is unbounded above, so $a<0$ would give $\phi(s)\to-\infty$; hence $a\ge1$. An increasing bijection of a set of order type $\omega$ onto itself is the identity, since by induction it fixes the $k$-th smallest element for every $k$.
\end{proof}

\begin{proposition}\label{prop:additive}
Let $S\subseteq\R$. Then $\End(S)$ is closed under pointwise addition if and only if $S+S\subseteq S$. In that case $(\End(S),+,\circ)$ is a right near-semiring, that is, $(\phi+\psi)\circ\chi=\phi\circ\chi+\psi\circ\chi$; the left distributive law $\chi\circ(\phi+\psi)=\chi\circ\phi+\chi\circ\psi$ holds exactly when $\chi(0)=0$, so the maps with $\phi(0)=0$ form a semiring (without zero if $0\notin S$).
\end{proposition}

\begin{proof}
If $S+S\subseteq S$, then $(\phi+\psi)(s)=\phi(s)+\psi(s)\in S$. Conversely, for $t,t'\in S$ the constant maps $c_t$ and $c_{t'}$ lie in $\End(S)$, and $c_t+c_{t'}=c_{t+t'}$. The distributive laws are a direct computation: if $\chi(x)=cx+e$, then $\chi\circ(\phi+\psi)$ and $\chi\circ\phi+\chi\circ\psi$ differ by the constant $e$.
\end{proof}

\begin{example}\label{ex:monoids}
If $S$ is a nonzero submonoid of $(\Z_{\ge0},+)$, for instance a numerical semigroup such as $\langle3,5\rangle=\{0,3,5,6,8,9,10,\dots\}$, then
\[
\End^*(S)=\{x\mapsto ax+b:\ a\in\N,\ b\in S\},
\]
a semidirect product $S\rtimes(\N,\times)$ with $(a,b)\circ(a',b')=(aa',ab'+b)$. Indeed $a\in\N$ by Proposition~\ref{prop:basic}, $b=\phi(0)\in S$, and conversely $as+b=s+\dots+s+b\in S$. If $S=d\Z$ with $d\ge1$, then $\End^*(S)=\{x\mapsto ax+b:\ a\in\Z\setminus\{0\},\ d\mid b\}$ and $\Aut(S)\cong D_\infty$.
\end{example}

\subsection{Sparse sets}\label{sec:sparse-sets}

\begin{theorem}\label{thm:sparse}
Let $S\subseteq\Z$ be infinite and bounded below, with elements $s_1<s_2<\cdots$, and assume that $s_{k+1}-s_k\to\infty$. Then:
\begin{enumerate}[label=(\alph*),leftmargin=*]
\item every self-map of $S$ is determined by its slope, and the only self-map of slope $1$ is the identity;
\item $\End^*(S)$ is commutative, and if $\End^*(S)\ne\{\id\}$, there is $\omega\in\Q$ with $\omega\le\min S$ such that every self-map is the homothety $\phi(x)=\omega+a(x-\omega)$;
\item consequently $\phi\mapsto a$ is an isomorphism of $\End^*(S)$ onto a submonoid $A(S)$ of $(\N,\times)$;
\item every submonoid $A\ne\{1\}$ of $(\N,\times)$ occurs up to isomorphism: for $S_A=\{d^2:\ d\in A\}$ one has $\omega=0$ and $A(S_A)=\{d^2:\ d\in A\}\cong A$.
\end{enumerate}
\end{theorem}

\begin{proof}
(a) Suppose that $\phi(x)=ax+b$ and $\psi(x)=ax+b'$ are self-maps with $b<b'$; here $a\ge1$ by Proposition~\ref{prop:basic}. For every $s\in S$ the numbers $as+b<as+b'$ belong to $S$ and differ by $b'-b$. Letting $s\to\infty$ produces arbitrarily large elements of $S$ at distance $b'-b$, contradicting $s_{k+1}-s_k\to\infty$. Comparing a self-map of slope $1$ with the identity gives the second statement.

(b) $\phi\circ\psi$ and $\psi\circ\phi$ have the same slope, so they coincide by (a). If $\phi$ and $\psi$ commute and have slopes different from $1$, with fixed points $\omega_\phi$ and $\omega_\psi$, then $\phi(\omega_\psi)=\phi(\psi(\omega_\psi))=\psi(\phi(\omega_\psi))$, so $\phi(\omega_\psi)$ is fixed by $\psi$; hence $\phi(\omega_\psi)=\omega_\psi$, and $\omega_\psi=\omega_\phi$. By (a) and Proposition~\ref{prop:basic}, all self-maps other than the identity have slope at least $2$, so they share a fixed point $\omega=b/(1-a)$, which is rational. If some $s\in S$ satisfied $s<\omega$, then $\phi^k(s)=\omega+a^k(s-\omega)\to-\infty$, which is impossible.

(c) follows from (a) and (b).

(d) The set $S_A$ is infinite and bounded below, and its gaps tend to infinity. For $c\in A$ the map $x\mapsto c^2x$ preserves $S_A$; taking $c\ge2$ shows that the common centre is $\omega=0$. Hence the self-maps of $S_A$ are the maps $x\mapsto\alpha x$ with $\alpha S_A\subseteq S_A$. As $1\in S_A$, this forces $\alpha=d^2$ with $d\in A$, and conversely every such $\alpha$ works.
\end{proof}

Theorem~\ref{thmC} consists of parts (a) and (b). For the polygonal numbers the centre is $\omega=\mu_s$, the minimum of $P_s$ on $\R$, and $A(\cP_s)=\{x^2:\ x\equiv1\pmod{M_s}\}$ by Theorem~\ref{thm:polygonal}. For $\{n^2+1:\ n\in\N\}$ one finds $\omega=1$ and the maps $x\mapsto\rho^2x+1-\rho^2$, and for $\{2^k:\ k\ge0\}$ one finds $\omega=0$ and $A=\{2^j:\ j\ge0\}$. The set $\{F_n:\ n\ge1\}$ of Fibonacci numbers, the tetrahedral numbers and the square pyramidal numbers have no self-map other than the identity; for the Fibonacci numbers, a homothety of integral ratio $a\ge2$ would force $F_{m_n}/F_n\to a$ along the images $F_{m_n}$ of the $F_n$, whereas such ratios accumulate only at the numbers $\tau^k$, $k\in\Z$, where $\tau=\frac{1+\sqrt5}{2}$, and these are irrational for $k\ne0$.

In the centred coordinate $X=x-\omega$, every self-map becomes linear, $X\mapsto aX$. This settles the question of when the sum of two self-maps is a self-map.

\begin{proposition}\label{prop:sums}
Let $S$ be as in Theorem~\ref{thm:sparse} with $\End^*(S)\ne\{\id\}$, let $\omega$ be its centre, and let $\phi,\psi$ be self-maps of $S$ with slopes $a,a'$.
\begin{enumerate}[label=(\alph*),leftmargin=*]
\item If $\omega\ne0$, then $\phi+\psi$ is not a self-map of $S$.
\item If $\omega=0$, then $\phi+\psi$ is a self-map of $S$ if and only if $a+a'\in A(S)$.
\item For $\lambda\in\R$, the affine combination $\lambda\phi+(1-\lambda)\psi$ is a self-map of $S$ if and only if $\lambda a+(1-\lambda)a'\in A(S)$.
\end{enumerate}
\end{proposition}

\begin{proof}
We have $\phi+\psi=(a+a')x+(2-a-a')\omega$, of slope $a+a'\ge2$ and with fixed point $\omega\frac{a+a'-2}{a+a'-1}$. By Theorem~\ref{thm:sparse}, a self-map of slope at least $2$ has fixed point $\omega$, and the unique self-map of slope $c\in A(S)$ is $x\mapsto\omega+c(x-\omega)$. The fixed point of $\phi+\psi$ equals $\omega$ only if $\omega=0$, which gives (a) and (b). For (c), $\lambda\phi+(1-\lambda)\psi$ is the homothety about $\omega$ of slope $\lambda a+(1-\lambda)a'$.
\end{proof}

Thus no sum of two self-maps of $\cP_s$ is a self-map when $s\ne4$; for example, $(9T+1)+(9T+1)=18T+2$ equals $20$ for $T=1$. For the squares, $x\mapsto m^2x$ plus $x\mapsto n^2x$ is a self-map exactly when $m^2+n^2$ is a square, as in $9x+16x=25x$, and for $\{2^k\}$ the sum $\phi+\psi$ is a self-map exactly when $\phi=\psi$. Affine combinations, on the other hand, can be self-maps: $\frac12\bigl(T+(49T+6)\bigr)=25T+3$, because $\frac12(1+49)=25$.

\subsection{The size of the image}\label{sec:image}
By Proposition~\ref{prop:basic}(d), the identity is the only surjective self-map of an infinite set $S\subseteq\Z$ that is bounded below, so it is natural to measure the size of the image.

\begin{proposition}\label{prop:image}
Let $S\subseteq\N$ be infinite, let $S(N)=\lvert S\cap[1,N]\rvert$, and let $\phi(x)=ax+b$ be a self-map of $S$ with $a\ge2$.
\begin{enumerate}[label=(\alph*),leftmargin=*]
\item If $S(N)\sim cN^\delta$ for some $c,\delta>0$, then $\lvert\phi(S)\cap[1,N]\rvert/S(N)\to a^{-\delta}$.
\item If $S\setminus\phi(S)$ is finite, then $S(N)=O(\log N)$.
\end{enumerate}
\end{proposition}

\begin{proof}
Since $\phi$ is injective and increasing, $\lvert\phi(S)\cap[1,N]\rvert=S\bigl((N-b)/a\bigr)+O(1)$, which gives (a). For (b), let $k=\lvert S\setminus\phi(S)\rvert$. Then $S(N)\le S\bigl((N+\lvert b\rvert)/a\bigr)+k$, and iterating this inequality $O(\log N)$ times gives $S(N)=O(\log N)$.
\end{proof}

For the polygonal numbers, $S(N)\sim\sqrt{2N/(s-2)}$, and by Theorem~\ref{thm:polygonal} the image of $E_{s,x}$ is $\{P_s(m):\ m\equiv-jp_s\pmod x,\ m>-jp_s\}$, of relative density $1/x$. Thus no self-map of slope at least $2$ of a set of polynomial growth has an image with finite complement, whereas for $\{2^k:\ k\ge0\}$ the map $x\mapsto2x$ misses only the element $1$. The condition on the slope cannot be dropped: the translation $x\mapsto x+1$ maps $\N$ onto $\N\setminus\{1\}$.

\subsection{Triangular numbers as a free orbit}\label{sec:freeorbit}
Let $\cT_0=\cT\cup\{0\}$. The following proposition identifies the triangular numbers, together with their self-maps, with the multiplicative monoid of odd positive integers.

\begin{proposition}\label{prop:freeorbit}
\begin{enumerate}[label=(\alph*),leftmargin=*]
\item The self-maps of $\cT_0$, and also those of $\cT$, are the Euler maps $E_u$ with $u$ odd, and $E_u\circ E_v=E_{uv}$. Hence both monoids are isomorphic to the multiplicative monoid of odd positive integers.
\item The map $u\mapsto E_u(0)=T_{(u-1)/2}$ is a bijection from the odd positive integers onto $\cT_0$, and $T\mapsto\sqrt{8T+1}$ carries the action of $E_u$ on $\cT_0$ to multiplication by $u$. Thus $\cT_0$ is a free orbit of its monoid of self-maps, generated by $0$.
\item $E_u(\cT_0)=\{T\in\cT_0:\ u\mid\sqrt{8T+1}\}$. Hence $E_u(\cT_0)\cap E_v(\cT_0)=E_{\mathrm{lcm}(u,v)}(\cT_0)$, and $E_u(\cT_0)\supseteq E_v(\cT_0)$ if and only if $u\mid v$; the images form a lattice isomorphic to the odd positive integers ordered by divisibility, with the order reversed.
\end{enumerate}
\end{proposition}

\begin{proof}
For (a), a self-map of $\cT_0$ or of $\cT$ has slope at least $1$ by Proposition~\ref{prop:basic}(d), so it is some $E_u$ by Corollary~\ref{cor:strong}; conversely each $E_u$ preserves $\cT_0$ and $\cT$. Everything else follows from the identity $8E_u(T)+1=u^2(8T+1)$ and the fact that $T\mapsto\sqrt{8T+1}$ is a bijection from $\cT_0$ onto the odd positive integers.
\end{proof}

\subsection{Factorization of self-maps}\label{sec:factorization}
By Proposition~\ref{prop:basic}(d), Theorem~\ref{thm:polygonal} and Remark~\ref{rem:poly}, the self-maps of $\cP_s$ are the maps $E_{s,x}$ with $x\equiv1\pmod{M_s}$, and $E_{s,x}\circ E_{s,y}=E_{s,xy}$. Hence
\[
\End^*(\cP_s)\;\cong\;H_{M_s},\qquad\text{where } H_M=\{x\in\N:\ x\equiv1\pmod{M}\}\ \text{ under multiplication.}
\]
Monoids of this kind are called arithmetical congruence monoids \cite{BCCM,CLMZ}; the classical Hilbert monoid is $H_4=\{1,5,9,13,\dots\}$. More generally, for $a,M\in\N$ with $a^2\equiv a\pmod M$ the set $H_{a,M}=\{1\}\cup\{x\in\N:\ x\equiv a\pmod M\}$ is closed under multiplication. If $a\not\equiv1\pmod M$, then $H_{a,M}$ is not a Krull monoid \cite[Thm.~1]{HalterKoch}; indeed, it is not completely integrally closed: the element $u=1+M$ of its quotient group satisfies $a\,u^n\in H_{a,M}$ for every $n\ge1$, although $u\notin H_{a,M}$. Only the Krull case $a=1$ arises here. For the factorization theory of arithmetical congruence monoids in general see \cite{BaginskiChapman,BCCM,CLMZ}; congruence monoids in Dedekind domains are often, but not always, weakly Krull monoids \cite[Examples~5.7.2]{GKR}. For $s=3$ we recover Proposition~\ref{prop:freeorbit}(a), and for $s=4$ the self-maps are the maps $N\mapsto m^2N$, so that $\End^*(\cP_4)\cong(\N,\times)$. We call a self-map \emph{irreducible} if it is not the identity and is not a composition of two self-maps other than the identity.

\begin{theorem}\label{thm:factorization}
Let $s\ge3$. Every self-map of $\cP_s$ is a composition of irreducible self-maps, and:
\begin{enumerate}[label=(\alph*),leftmargin=*]
\item the factorization into irreducible self-maps is unique up to order for every self-map if and only if $s\in\{3,4\}$;
\item all factorizations of each self-map have the same number of factors if and only if $s\in\{3,4,5,6,8\}$.
\end{enumerate}
\end{theorem}

\begin{proof}
It suffices to prove the corresponding statements for $H_M$ with $M=M_s$. These statements are known \cite[\S3]{BaginskiChapman}, \cite{HalterKoch,JamesNiven}, and we recall the argument. Factorizations into atoms exist because an element of $H_M$ that is not an atom is a product of two smaller elements. Let $F$ be the free commutative monoid on the primes not dividing $M$. Then $H_M$ is a saturated submonoid of $F$ (if $x,y\in H_M$ and $x\mid y$ in $F$, then $y/x\in H_M$), and the embedding is a divisor theory: every $\alpha\in F$ is the greatest common divisor of $\alpha q$ and $\alpha q'$, where $q,q'\nmid\alpha M$ are distinct primes with $q\equiv q'\equiv\alpha^{-1}\pmod M$. The class group $\mathsf q(F)/\mathsf q(H_M)$ is $(\Z/M\Z)^\times$. Thus $H_M$ is a Krull monoid with class group $(\Z/M\Z)^\times$, and by Dirichlet's theorem every class contains infinitely many primes. For such monoids, factorization is unique if and only if the class group is trivial, and all factorizations of each element have the same length if and only if the class group has at most two elements; the latter is Carlitz's theorem \cite{Carlitz}, which holds for Krull monoids of this kind \cite{GHK}. Here the class group has order $\varphi(M)$, where $\varphi$ is Euler's function, and $\varphi(M)=1$ if and only if $M\in\{1,2\}$, while $\varphi(M)\le2$ if and only if $M\in\{1,2,3,4,6\}$. By Corollary~\ref{cor:divisibility}, $M_s\in\{1,2\}$ if and only if $s\in\{3,4\}$, and $M_s\in\{1,2,3,4,6\}$ if and only if $s\in\{3,4,5,6,8\}$.
\end{proof}

The appeal to the theory of Krull monoids can be replaced by explicit factorizations, and this is how the formal verification proceeds; Dirichlet's theorem provides the primes used below. If $M\in\{1,2\}$, the irreducible elements of $H_M$ are the primes not dividing $M$, and factorization is unique. If $M\notin\{1,2\}$, choose distinct primes $p_1,p_2,p_3,p_4\equiv-1\pmod M$; then $(p_1p_2)(p_3p_4)=(p_1p_3)(p_2p_4)$ are two different factorizations into irreducible elements of $H_M$. If $\varphi(M)\ge3$, there are units $g,h$ of $\Z/M\Z$ with $g\ne1$, $h\ne1$ and $gh\ne1$; for primes $p,q,r,p',q',r'$ in the classes $g,h,(gh)^{-1},g^{-1},h^{-1},gh$, the element $(pqr)(p'q'r')=(pp')(qq')(rr')$ has factorizations into irreducible elements of lengths $2$ and $3$. Finally, let $M\in\{1,2,3,4,6\}$, so that every unit of $\Z/M\Z$ is $\pm1$, and give each prime factor $\equiv1\pmod M$ weight $2$ and every other prime factor weight $1$. If an irreducible element $x$ of $H_M$ has a prime factor $p\equiv1\pmod M$, then $x=p$; otherwise all prime factors of $x$ are $\equiv-1\pmod M$, and since the product of any two of them lies in $H_M$, $x$ has exactly two prime factors. Thus every irreducible element of $H_M$ has weight $2$, and all factorizations of an element of $H_M$ have the same length. It is not enough here that every unit of $\Z/M\Z$ squares to $1$: for $M=8$, the element $105=3\cdot5\cdot7$ of $H_8$ is irreducible and has weight $3$.

Theorem~\ref{thmD} is Theorem~\ref{thm:factorization} together with the isomorphism $\End^*(\cP_s)\cong H_{M_s}$.

\begin{example}\label{ex:hilbert}
For hexagonal numbers ($M_6=4$) the self-maps are $E_{6,x}(N)=x^2N+\frac{x^2-1}{8}$ with $x\equiv1\pmod4$, and
\[
E_{6,9}\circ E_{6,49}=E_{6,21}\circ E_{6,21},\qquad\text{i.e.}\qquad 81\,(2401N+300)+10=441\,(441N+55)+55,
\]
where $E_{6,9}$, $E_{6,21}$ and $E_{6,49}$ are irreducible because $3,7\not\equiv1\pmod4$; this is Hilbert's example $441=9\cdot49=21\cdot21$ in $H_4$. For heptagonal numbers ($M_7=10$) even the numbers of factors differ: $E_{7,81}\circ E_{7,2401}=E_{7,21}\circ E_{7,21}\circ E_{7,21}\circ E_{7,21}$, with all factors irreducible. For triangular numbers, by contrast, the irreducible self-maps are the Euler maps $E_p$ with $p$ an odd prime, and every self-map factors uniquely into them.
\end{example}

\subsection{Local obstructions}\label{sec:local}
Theorem~\ref{thm:sparse} does not apply to sets such as the squarefree numbers or the sums of two squares, whose gaps do not tend to infinity. For such sets, and for the primes, congruence arguments take its place, as in the proof of Theorem~\ref{thm:two}.

\begin{proposition}\label{prop:local}
\begin{enumerate}[label=(\alph*),leftmargin=*]
\item The only self-map of the set of primes is the identity.
\item The only self-map of the set of squarefree positive integers is the identity.
\item The self-maps of $\Sigma=\{x^2+y^2:\ x,y\in\Z\}$ are the maps $x\mapsto ax$ with $a\in\Sigma$ and $a\ge1$. This monoid is free, with basis consisting of $2$, the primes $p\equiv1\pmod4$ and the squares $q^2$ of the primes $q\equiv3\pmod4$.
\end{enumerate}
\end{proposition}

\begin{proof}
In each case a self-map $\phi(x)=ax+b$ has $a\in\N$ and $b\in\Z$ by Proposition~\ref{prop:basic}.

(a) Suppose that $ap+b$ is prime for every prime $p$, and that $b\ne0$. Choose a prime $\ell\nmid ab$. By Dirichlet's theorem there are infinitely many primes $p\equiv-ba^{-1}\pmod\ell$; for them $\ell\mid ap+b$, while $ap+b>\ell$ once $p$ is large, a contradiction. So $b=0$, and then $a=1$.

(b) Suppose that $b\ne0$, and choose a prime $\ell\nmid a$ with $\ell^2\nmid b$. Squarefree numbers occur in every residue class modulo $\ell^2$ other than $0$, so there is a large squarefree $s$ with $as+b\equiv0\pmod{\ell^2}$, a contradiction. So $b=0$. If $a>1$, then for a prime divisor $p$ of $a$ the number $ap$ is divisible by $p^2$, although $p$ is squarefree; hence $a=1$.

(c) The maps $x\mapsto ax$ with $a\in\Sigma$ preserve $\Sigma$ by the Brahmagupta--Fibonacci identity. Conversely, $b=\phi(0)\in\Sigma$. If $b\ne0$, choose a prime $\ell\equiv3\pmod4$ with $\ell\nmid ab$. As in the proof of Theorem~\ref{thm:two}, there is $m\in\Sigma$ with $am+b\equiv\ell\pmod{\ell^2}$, and then $am+b\notin\Sigma$, a contradiction. So $b=0$, and $a=\phi(1)\in\Sigma$. A positive integer lies in $\Sigma$ if and only if each prime $q\equiv3\pmod4$ divides it to an even power, which gives the description of the monoid.
\end{proof}

\subsection*{Acknowledgments}
The literature search and the Lean formalization were prepared with the assistance of the AI system Claude (Anthropic). The author takes full responsibility for the content of this paper.

\end{document}